\documentclass[12pt,a4paper,reqno]{amsproc}

\makeatletter

\usepackage{datetime}
\usepackage[pdfborder={0 0 0}]{hyperref}
  \hypersetup{
    citebordercolor=.1 .6 1,
    colorlinks = false
}

\usepackage{upref,cases}
\hypersetup{citebordercolor=.1 .6 1}

\usepackage[T1]{fontenc}
\usepackage[utf8]{inputenc}

\usepackage[final]{graphicx}

\IfFileExists{mmymtpro2.sty}{\usepackage[subscriptcorrection]{mymtpro2}
	\newcommand{\id}{\mathbf{1}}
	
  \def\cup{\cupprod}
  
  \def\bigcup{\bigcupprod}
  
  \def\bigcupdisjoint{\mathop{\kern10pt\raisebox{4pt}{$\cdot$}\kern-12pt\bigcup}\limits}
}%
	{\usepackage{amssymb,bm,dsfont,fourier,times}
	 
	 \newcommand{\id}{\ensuremath{\mathds{1}}}
	\let\wtilde\widetilde
}

\let\epsilon\varepsilon

\renewcommand{\Im}{\ensuremath{\operatorname{Im}}}

\DeclareMathOperator{\Tr}{Tr}

\newcommand{\abs}[1]{\ensuremath{\left\vert #1 \right\vert}}

\DeclareMathOperator{\supp}{supp}
\DeclareMathOperator{\dist}{dist}

\numberwithin{equation}{section}
\newtheoremstyle{ttheorem}%
       {1.8ex\@plus1ex}                
       {2.1ex\@plus1ex\@minus.5ex}      
       {\itshape}           
       {0pt}                   
       {\bfseries}          
       {.}                  
       {.5em}               
       {}                

\newtheoremstyle{ddefinition}%
       {1.8ex\@plus1ex}                
       {2.1ex\@plus1ex\@minus.5ex}      
       {}           
       {0pt}                   
       {\bfseries}           
       {.}                  
       {.5em}               
       {}                

\newtheoremstyle{rremark}%
       {1.8ex\@plus1ex}                
       {2.1ex\@plus1ex\@minus.5ex}      
       {\normalfont}        
       {0pt}                   
       {\bfseries}           
       {.}                  
       {.5em}               
       {}                   

\theoremstyle{ttheorem}
\newtheorem{theorem}{Theorem}[section]
\newtheorem{lemma}[theorem]{Lemma}
\newtheorem{proposition}[theorem]{Proposition}
\newtheorem{corollary}[theorem]{Corollary}

\theoremstyle{ddefinition}

\theoremstyle{rremark}
\newtheorem{remark}[theorem]{Remark}
\newtheorem{myremarks}[theorem]{Remarks}
\newtheorem{myexamples}[theorem]{Examples}

\newenvironment{remarks}{\begin{myremarks}\begin{nummer}}%
    {\end{nummer}\end{myremarks}}
    {\end{nummer}\end{myexamples}}

\newcounter{numcount}
\newcommand{\labelnummer}{(\roman{numcount})}%

\makeatletter
\providecommand{\showkeyslabelformat}[1]{\relax}        
\let\mysaveformat\showkeyslabelformat                   %
\def\myformat#1{\raisebox{-1.5ex}{\mysaveformat{#1}}}   %

\newenvironment{nummer}%
  {\let\curlabelspeicher\@currentlabel%
    \begin{list}{\textup{\labelnummer}}%
      {\usecounter{numcount}\leftmargin0pt%
        \topsep0.5ex\partopsep2ex\parsep0pt\itemsep0ex\@plus1\p@%
        \labelwidth2.5em\itemindent3.5em\labelsep1em%
      }%
    \let\saveitem\item%
    \def\item{\saveitem%
      \def\@currentlabel{\curlabelspeicher\kern.1em\labelnummer}}%
    \let\savelabel\label%
    \def\label##1{{\ifnum\thenumcount=1\let\showkeyslabelformat\myformat\fi\savelabel{##1}}%
										{\def\@currentlabel{\labelnummer}%
									 	\let\showkeyslabelformat\@gobble
									 	\savelabel{##1item}%
										}%
	   							}%
  }{\end{list}}%

  {\let\curlabelspeicher\@currentlabel%
    \begin{list}{\textup{\labelnummer}}%
      {\usecounter{numcount}\leftmargin0pt%
        \topsep0.5ex\partopsep2ex\parsep0pt\itemsep0ex\@plus1\p@%
        \labelwidth2.5em\itemindent0em\labelsep1em%
        \leftmargin2.5em}%
    \let\saveitem\item%
    \def\item{\saveitem%
      \def\@currentlabel{\curlabelspeicher\kern.1em\labelnummer}}%
    \let\savelabel\label%
    \def\label##1{{\ifnum\thenumcount=1\let\showkeyslabelformat\myformat\fi\savelabel{##1}}%
										{\def\@currentlabel{\labelnummer}%
									 	\let\showkeyslabelformat\@gobble
									 	\savelabel{##1item}%
										}%
    							}%
  }{\end{list}}%

\usepackage{enumitem}

\let\OldItem\item
\newcommand{\MyItem}[2][]{}%
\def\section{\@startsection{section}{1}%
  \z@{1.3\linespacing\@plus\linespacing}{.5\linespacing}%
  {\normalfont\bfseries\centering}}

\def\subsection{\@startsection{subsection}{2}%
  \z@{.8\linespacing\@plus.5\linespacing}{-1em}%
  {\normalfont\bfseries}}

\def\nlsubsection{\@startsection{subsection}{2}%
  \z@{.8\linespacing\@plus.5\linespacing}{.1ex}%
  {\normalfont\bfseries}}

\let\@afterindenttrue\@afterindentfalse%

\renewenvironment{proof}[1][\proofname]{\par \normalfont
  \topsep6\p@\@plus6\p@ \trivlist 
  \item[\hskip\labelsep\scshape
    #1\@addpunct{.}]\ignorespaces
}{%
  \qed\endtrivlist
}

\def\ps@firstpage{\ps@plain
  \def\@oddfoot{\normalfont\scriptsize \hfil\thepage\hfil
     \global\topskip\normaltopskip}%
  \let\@evenfoot\@oddfoot
  \def\@oddhead{
    \begin{minipage}{\textwidth}
      \normalfont\scriptsize
      \emph{\insertfirsthead}
    \end{minipage}}
  \let\@evenhead\@oddhead 
}

\def\insertfirsthead{}

\def\@cite#1#2{{%
 \m@th\upshape\mdseries[{#1}{\if@tempswa, #2\fi}]}}
\newcommand{\C}{\mathbb{C}}
\newcommand{\N}{\mathbb{N}}
\newcommand{\PP}{\mathbb{P}}
\newcommand{\R}{\mathbb{R}}
\newcommand{\Z}{\mathbb{Z}}

\renewcommand{\leq}{\leqslant}
\renewcommand{\geq}{\geqslant}
\providecommand{\wtilde}[1]{\widetilde{#1}}

\let\<\langle
\let\>\rangle

\providecommand{\abs}[1]{\lvert#1\rvert}

\newcommand{\upd}{\mathrm{d}}
\renewcommand{\d}{\upd}   

\newcommand{\hairspace}{\kern .04167em}

\newcommand{\Zd}{\mathbb{Z}^d}

\DeclareMathOperator{\TextIm}{Im}

\renewcommand{\Im}{\TextIm}

\newcommand{\beq}{\begin{equation}}
\newcommand{\eeq}{\end{equation}}
\newcommand{\be}{\begin}

\newcommand{\E}{\mathbb E}

\def\clap#1{\hbox to 0pt{\hss#1\hss}}

\makeatother
\usepackage{todonotes}

\pdfmapfile{=fourier-mex.map}

\begin{document}

\title[Lifshitz tails for the fractional Anderson model]{Lifshitz tails for the fractional Anderson model}

\author[M.\ Gebert]{Martin Gebert}
\address[M.\ Gebert]{ Department of Mathematics, University of California, Davis, Davis, CA 95616, USA}
\email{mgebert@math.ucdavis.edu}

\author[C.\ Rojas-Molina]{Constanza Rojas-Molina}
\address[C.\ Rojas-Molina]{Laboratoire AGM, Dpt. math\'ematiques, CY Cergy Paris Universit\'e,
2 av. Adolphe Chauvin, 95302 Cergy-Pontoise, France}
\email{crojasmo@u-cergy.fr}

\begin{abstract}
We consider the $d$-dimensional fractional Anderson model $(-\Delta)^\alpha+ V_\omega$ on $\ell^2(\Z^d)$ where $0<\alpha\leq 1$. Here $-\Delta$ is the negative discrete Laplacian and $V_\omega$ is the random Anderson potential consisting of iid random variables. We prove that the model exhibits Lifshitz tails at the lower edge of the spectrum with exponent $ d/ (2\alpha)$. To do so, we show among other things that the non-diagonal matrix elements of the negative discrete fractional Laplacian are negative and satisfy the two-sided bound
$$
\frac{c_{\alpha,d}}{|n-m|^{d+2\alpha}} \leq -(-\Delta)^\alpha(n,m)\leq \frac{C_{\alpha,d}}{|n-m|^{d+2\alpha}}
$$
for positive constants $c_{\alpha,d}$, $C_{\alpha,d}$ and all $n\neq m\in\Z^d$.
\end{abstract}

\maketitle

\section{Introduction}
Fractional operators are non-local operators that arise in the study of systems with long-range interactions in connection with anomalous transport and in some cases with L\'evy processes, see e.g. \cite{MR3074586,Riascos_2018}. For this reason, they have been the subject of increasing interest in recent years, where they have been studied in both the discrete and continuous setting. The latter case has been well studied in the literature, see e.g. \cite{MR1054115, MR3613319,MR3916700} and references therein, while the discrete setting and the arising fractional dynamics has attracted greater interest recently, see e.g. \cite{MR3787555,Padgett_Liaw_etal19}.

In this note, we study the Integrated Density of States (IDS) of the discrete $\alpha$-fractional  Laplacian perturbed by a random Anderson potential, $ H_\alpha=(-\Delta)^\alpha+ V_\omega$ acting on $\ell^2(\Z^d)$ where $0<\alpha\leq 1$ and $d\in\N$ is the space dimension. We show, in particular, that the IDS of this model exhibits exponential decay near the lower spectral band edge, a phenomenon known as Lifshitz tails, see \cite{MR1223779,MR2509110,MR3364516} and references therein. It turns out that in our setting the Lifshitz exponent is given by $d/(2\alpha)$. To our understanding, this behaviour has only been proved so far in some continuous cases
  in the work of \^{O}kura \cite{MR736015}, and of Kaleta and Pietruska-Paluba \cite{MR1109472,kaleta2019lifschitz,kaleta2019lifschitzb}. The aforementioned studies are based on  probabilistic tools revolving around $\alpha$-stable L\'evy processes and corresponding Feynman-Kac formulas. Contrary to that, our proof in the discrete case is functional analytic in nature. It relies on Dirichlet-Neumann bracketing and operator monotonicity of the function $x\mapsto x^\alpha$ for $x\geq 0$ and $0<\alpha\leq 1$. 

One of the main features of the continuous fractional Laplacian is the slow (polynomial) off-diagonal decay of its kernel components, which is expected to be true also in the discrete setting. This was for example sketched in \cite{MR3600570} for arbitrary dimensions using a Tauberian argument, and for one-dimensional models it was shown in e.g. \cite{MR3787555}. This underlines the interpretation of the discrete operator $H_\alpha$ as an operator with long-range slowly decaying off-diagonal matrix elements perturbed by an on-site random potential.
In dimension $d=1$, similar models have been studied in \cite{MR1668075,MR3973892}, where the slow decay of the off-diagonal matrix elements is shown to have several consequences on the spectral and dynamical properties of the model.
In this note, we also provide a rigorous proof of the polynomial off-diagonal decay of the  matrix elements of the discrete fractional Laplacian $(-\Delta)^\alpha$ in arbitrary dimensions $d\geq 1$ with power $d+2\alpha$. This shows that the kernel of the discrete fractional Laplacian has the same off-diagonal decay as its continuous analogue see Theorem \ref{Lemma}. 

The paper is organised as follows: in the next section we state the model and results.  In Section \ref{s:fin_vol_approxIDS} we show that the IDS can be obtained as the limit of spectral projections of a finite-volume version of the Hamiltonian $H_{\alpha}$. In Section \ref{s:thm_LT} we show the fractional Lifshitz tails result. Finally, in Section \ref{s:off_diag_decay} we show the aforementioned off-diagonal decay of matrix elements of the fractional Laplacian $(-\Delta)^\alpha$.

\smallskip

\emph{Keywords: fractional Laplacian, Anderson model, random Schr\"odinger operators, Lifshitz tails, integrated density of states}

\section{Model and results} \label{s:mod_res}

In the following, $(\delta_n)_{n\in\Z^d}$ denotes the canonical orthonormal basis of $\ell^{2}(\Z^d)$, and $|\delta_n\>\<\delta_n|$ the orthogonal projection onto the subspace spanned by the vector $\delta_n$,  $n\in\Z^d$, where $d$ is the space dimension. For an operator $A$ acting on $\ell^2(\Z^d)$ and $n,m\in \Z^d$ we denote by 
\beq
A(n,m):=\<\delta_n,A\,\delta_m\>
\eeq
 the matrix elements of  $A$. Furthermore, for $p\in [1,\infty]$ we denote the $\ell^p$-norm of a vector  $x\in\Z^d$ by $|x|_p=\big(\sum_{i=1}^{d}\abs{x_i}^p\big)^{\frac 1 p}$. For $p=2$, the Euclidean norm on $\Z^d$, we use  the short hand notation $|\cdot|$. Given $L\in\N$, we denote by $\Lambda_L:=[-L,L]^d$ the box in $\Z^d$.

Let $0<\alpha\leq 1$, $\lambda>0$ and $d\in\N$. We consider the discrete fractional Anderson model of the form
\begin{equation}
\label{eq:TheOperator}
 H_{\alpha}:= (-\Delta)^\alpha +\lambda V_\omega
\end{equation}
acting on the Hilbert space $\ell^{2}(\Z^d)$, where $(-\Delta)^\alpha$ and $ V_{\omega}$ are subject to the following:
\be{MyDescription}
\item[(A)]{The operator $(-\Delta)^\alpha$ is the discrete fractional Laplacian defined by the functional calculus, where $-\Delta$ denotes the
 the discrete Laplacian on $\ell^2(\Z^d)$ given by $(-\Delta\varphi)(n):= \sum_{|n-m|_1=1} \big(\varphi(n)-\varphi(m)\big)$ for $\varphi\in\ell^2(\Z^d)$ and $n\in\Z^d$.
  \label{assA}}
\item[(B)]{ The random potential is given by $V_\omega := \sum_{n\in\Z^d} \omega_n \big|\delta_n\big\>\big\<\delta_n\big|$ with $\omega:=(\omega_n)_{n\in\Zd} \in \R^{\Z^{d}}$ being
	identically and independently distributed according to the Borel probability measure $\PP := \bigotimes_{\Z^{d}}P_{0}$
	on $\R^{\Z^{d}}$. The single-site probability measure $P_0$ is non-trivial  and we assume $0$ is the infimum of $\supp P_0$, the support of $P_0$ . We denote the corresponding expectation by $\E[\cdot]$.
 	\label{B}  }
\end{MyDescription}
As a consequence of the translation invariance of the unperturbed operator $(-\Delta)^\alpha$ and $(B)$, the operator $H_{\alpha}$ is ergodic in the usual sense \cite{MR1223779,MR2509110}. Hence, standard arguments imply that the spectrum $\text{spec}(H_\alpha)$ of $H_{\alpha}$ is deterministic and we have $\text{spec}(H_\alpha)=[0,(4d)^\alpha]+ \supp P_0$ for almost all $\omega\in\Omega$, see \cite[Thm. 3.9]{MR2509110}. Since we assume $0=\inf \supp P_0$, we have that $\inf \text{spec}(H_\alpha) =0$ almost surely.

In the following, we denote by $1_{(-\infty,E]}(H_\alpha)$  the spectral projection of $H_{\alpha}$ associated to the interval $(-\infty,E]$, and for $L\in\N$, we denote by $1_L$ the projection onto the box $\Lambda_L$. Then ergodicity of $H_\alpha$ also implies that, for almost all $\omega\in\Omega$, the limit
\beq \label{eq:def_ids}
\lim_{L\to\infty}\frac 1 {|\Lambda_L|}\Tr\big[ 1_{(-\infty,E]}(H_\alpha) 1_L\big] = \mathbb E\big[ \big\<\delta_0, 1_{(-\infty,E]}(H_\alpha)\delta_0\big\>\big] =:N_\alpha(E)
\eeq
exists for all $E\in\R$, where $|\Lambda_L|= (2L+1)^d$ is the volume of the box.
We call $N_\alpha(E)$ the integrated density of states of $H_\alpha$ (IDS).

For the standard Anderson model, i.e. $\alpha=1$,  an equivalent way of defining the IDS is by restricting the operator to finite volume and considering the limit of the so-called normalised eigenvalue counting function. This can be done for our model as well and is the subject of the next proposition. Given $L\in\N$, we denote by $H_{\alpha,L}$ the restriction of $H_\alpha$ to $\ell^{2}(\Lambda_L)$ given by
$H_{\alpha,L}:= 1_L H_\alpha 1_L$, where $1_L$ is the orthogonal projection onto $\Lambda_L$
as above.
\begin{proposition}\label{Proposition}
Let $0<\alpha\leq 1$ and let $H_{\alpha,L}$ be the finite-volume restriction of $H_\alpha$ to $\ell^{2}(\Lambda_L)$. Then, almost surely, the limit
\beq
\lim_{L\to\infty} \frac 1 {|\Lambda_L|} \Tr \big[1_{(-\infty,E]}(H_{\alpha.L})\big]
\eeq
exists for all $E\in\R$ and equals the IDS $N_\alpha$ at energy $E$, defined in \eqref{eq:def_ids}.
\end{proposition}

In order to prove Proposition \ref{Proposition}, we need the following bound on the off-diagonal decay of the matrix elements of $(-\Delta)^\alpha$ in arbitrary dimension $d\in \N$:

\begin{theorem}\label{Lemma}
Let $0<\alpha<1$. Then
\begin{enumerate}
\item[(i)]
For all $n,m\in\Z^d$ with $n\neq m$
\beq\label{eq:matele}
(-\Delta)^\alpha(n,m)<0
\eeq
and $(-\Delta)^\alpha(n,n)>0$.
\item[(ii)]
The limit 
\beq
\lim_{|n-m|\to\infty} |n-m|^{d+2\alpha} \big(-(-\Delta)^\alpha(n,m)\big)
 = 
K_{d,\alpha}>0
\eeq
exists with limit
\beq
K_{d,\alpha}= \frac{4^\alpha \Gamma(d/2+\alpha)}{\pi^{d/2} |\Gamma(-\alpha)|}>0. 
\eeq
\item[(iii)]
There exist constants $C_{d,\alpha}, c_{d,\alpha}>0$ depending only on $d$ and $\alpha$ such that for all $n,m\in\Z^d$ with $n\neq m$  
\beq \label{eq:decay}
\frac{c_{\alpha,d}}{|n-m|^{d+2\alpha}} \leq -(-\Delta)^\alpha(n,m)\leq \frac{C_{\alpha,d}}{|n-m|^{d+2\alpha}}.
\eeq
\end{enumerate}
\end{theorem}

\begin{remark}
In $d=1$, one can compute $(-\Delta)^{\alpha}(n,m)$ explicitly which results in a two-sided bound of the above form, see e.g. \cite{MR3787555}. While in higher dimensions this bound is generally accepted, we provide a rigorous proof and show that the decay of the fractional discrete Laplacian is indeed the same as in the continuous case, where this bound is well-known in arbitrary dimensions, see e.g. \cite{MR1054115}. The constant $K_{d,\alpha}$ is precisely the constant appearing in the definition of the continuous fractional Laplacian, see equation \eqref{eq:int_lap}. 
\end{remark}

Under additional assumptions on the regularity of the random variables, the off-diagonal decay in \eqref{eq:decay}  readily implies using  \cite[Thm. 3.1]{MR1244867}:

\begin{corollary}[Localisation at strong disorder]
Let $0<\alpha\leq 1$ and suppose that the measure  $P_0$ is absolutely continuous with respect to Lebesgue measure. Then, there exists a coupling $\lambda_0(d)>0$ such that for all $\lambda\geq \lambda_0(d)$  the spectrum of $(-\Delta)^\alpha + \lambda V_\omega$ consists of  point spectrum only.
\end{corollary}

\begin{remarks}
\item
The spatial decay of the eigenfunctions is related to the decay rate of the Green's function \cite{MR820340}. Due to the polynomial off-diagonal decay of the matrix elements of $(-\Delta)^\alpha$ we obtain polynomial decay of the Green's function of $H_\alpha$. This indicates that eigenfunctions of $H_\alpha$ decay at least polynomially.
\item
In $d=1$ it is conjectured that our model exhibits a phase transition when varying the fractional parameter $\alpha$: For $\alpha<\frac 1 2$, it is expected that $H_\alpha$ has continuous spectrum whereas for $\alpha>\frac 1 2$ it has point spectrum only, see \cite{MR1668075} and also \cite{MR1017742}. 
\end{remarks}

Our main result shows that the IDS $N_\alpha$  exhibits a fractional Lifshitz-tail behaviour at the lower edge of the almost-sure spectrum, which is $0$ under our assumptions.

\begin{theorem}[Fractional Lifshitz tails]\label{Theorem}
Let $0<\alpha\leq1$. The integrated density of states $N_\alpha$ of $H_\alpha$ satisfies
\beq
\lim_{E\searrow 0} \frac{\ln |\ln N_\alpha(E)|}{\ln E} \leq -\frac d {2\alpha} .
\eeq
Under the additional assumption that $P_0([0,\varepsilon))\geq C \varepsilon^\kappa$ for some $C,\kappa>0$ we obtain equality, i.e.
\beq
\lim_{E\searrow 0} \frac{\ln |\ln N_\alpha(E)|}{\ln E} = -\frac d {2\alpha} .
\eeq
\end{theorem}

\begin{remark}
Fractional Lifshitz tails are known in the continuous setting: they were first obtained by \^{O}kura in \cite{MR736015} for the continuous fractional Laplacian perturbed by Poissonian random potentials.
 The result also extends to Gaussian perturbations \cite{RM-P}. Similar results for Poissonian models on fractals were obtained in Kaleta and Pietruska-Paluba in \cite{{MR3860014}} and recently for the Anderson model on $\rm L^2(\R^d)$ in \cite{kaleta2019lifschitz,kaleta2019lifschitzb}.
\end{remark}

\section{Proof of Proposition \ref{Proposition}}\label{s:fin_vol_approxIDS}

\begin{proof}[Proof of Proposition \ref{Proposition}]
We restrict ourselves to $0<\alpha<1$ in the proof. 
Following \cite[Sec. 5.4]{MR2509110}, it suffices to show that
for all $z\in \C\setminus \R$
\beq\label{Pf:Prop:Eq0}
\lim_{L\to\infty}\frac 1 {|\Lambda_L|} \Tr\Big[ \frac 1 {H_{\alpha,L}-z} - 1_L \frac 1 {H_\alpha-z} 1_L\Big]  =0.
\eeq
Let $z\in \C\setminus \R$.  We estimate using the resolvent equation
\begin{align}
&\Big|\Tr\Big[ \frac 1 {H_{\alpha,L}-z} - 1_L \frac 1 {H_\alpha-z} 1_L\Big] \Big|\notag\\
&\leq
\sum_{n\in\Lambda_L} \Big|\frac 1 {H_{\alpha,L}-z} (n,n)- \frac 1 {H_\alpha-z}(n,n)\Big|\notag\\
&\leq
\sum_{n\in\Lambda_L}\sum_{k\in\Lambda_L} \sum_{m\in \Z^d}\Big|\frac 1 {H_{\alpha,L}-z} (n,k) \big( (-\Delta)_{L}^\alpha -(-\Delta)^\alpha\big)(k,m) \frac 1 {H_\alpha-z}(m,n)\Big|,\label{Pf:Prop:Eq1}
\end{align}
where $(-\Delta)^\alpha_L:= 1_L (-\Delta)^\alpha 1_L$.
Now $ (-\Delta)^\alpha_{L}- (-\Delta)^\alpha= 1_L (-\Delta)^\alpha 1_{L^c} + 1_{L^c} (-\Delta)^\alpha 1_L+1_{L^c} (-\Delta)^\alpha 1_{L^c}$, with $1_{L^c}$ denoting the projection onto $\R^d\setminus \Lambda_L$. The off-diagonal decay of the matrix elements stated in Theorem \ref{Lemma}
 and Cauchy-Schwarz inequality imply
\begin{align}
\eqref{Pf:Prop:Eq1} \leq
& \sum_{k\in\Lambda_L}\sum_{m\in \Z^d}\big| \big( (-\Delta)^\alpha_{L} -(-\Delta)^\alpha\big)(k,m)\big|
\big\|\frac 1 {H_{\alpha,L}-z}\delta_k\big\| \big\|\frac 1 {H_{\alpha}-z}\delta_m\big\|\notag\\
\leq&
\frac {C_{d,\alpha}} {|\Im z|^2}  \sum_{k\in\Lambda_L}\sum_{m\in\Lambda^c _L} \frac 1 {|k-m|^{d+2\alpha}}.
\label{pf:IDS:eq0}
\end{align}
Let $2\alpha<1$. In this case
we estimate the latter double sum by the corresponding integral, i.e.
\begin{align}
 \sum_{k\in\Lambda_L}\sum_{m\in\Lambda^c _L} \frac 1 {|k-m|^{d+2\alpha}} \leq \int_{\Lambda_{L+1/2}} \d x \int_{\Lambda_{L+1/2}^c} \d y\, \frac 1 {|x-y|^{d+2\alpha}},\label{pf:IDS:eq1}
\end{align}
where on the right hand side in slight abuse of notation $\Lambda_{L+1/2}=(-L- 1/2, L+1/2)^d\subset\R^d$ is the box in the continuum.
After a change of variables this reads
\begin{align}\label{pf:IDS:eq2}
\eqref{pf:IDS:eq1} = \Big(L+\frac 1 2\Big)^{d-2\alpha} \int_{\Lambda_1} \d x \int_{\Lambda_1^c} \d y\, \frac 1 {|x-y|^{d+2\alpha}},
\end{align}
which is finite for $2\alpha<1$. In order to see this, we note that $\Lambda_1^c\subset B^c_{\dist(x,\Lambda_1^c)}(x):=\R^d\setminus B_{\dist(x,\Lambda_1^c)}(x)$ for all $x\in (-1,1)^d$. Here $B_a(x)$ stands for the ball of radius $a>0$ around $x\in\R^d$ and $\dist(x,S)$ for the Euclidean distance of $x\in\R^d$ to a set $S\subset \R^d$. This and a change of variables imply
\begin{align}
\int_{\Lambda_1} \d x \int_{\Lambda_1^c} \d y\, \frac 1 {|x-y|^{d+2\alpha}}
&\leq \int_{\Lambda_1} \d x\, \int_{B^c_{\dist(x,\Lambda_1^c)}(x)} \d y\, \frac 1 {|x-y|^{d+2\alpha}}
\notag\\
&=\omega_{d-1}\int_{\Lambda_1} \d x\, \int_{\dist(x,\Lambda_1^c)}^\infty \d r\,r^{d-1} \frac 1 {r^{d+2\alpha}}\notag\\
&= \frac {\omega_{d-1}} {2\alpha} \int_{\Lambda_1} \d x\, \frac 1 {\dist(x,\Lambda_1^c)^{2\alpha}},
\end{align}
where $\omega_{d-1}$ is the surface area of the $d$-dimensional unit sphere.
We note that $\frac 1 {\dist(x,\Lambda_1^c)^{2\alpha}}\leq \frac 1{1-|x|_\infty}$.
Moreover, we denote by $\d S$ the surface measure. Hence, we bound the latter integral by
\begin{align}
\int_{\Lambda_1} \d x \frac 1 {\dist(x,\Lambda_1^c)^{2\alpha}}
\leq \int_0^1 \d r \int_{\partial\Lambda_r} \d S(k) \frac 1 {(1-r)^{2\alpha}}
= \int_0^1 \d r |\partial\Lambda_r|  \frac 1 {(1-r)^{2\alpha}},
\end{align}
where the latter is finite for $2\alpha<1$.

Now let $2\alpha\geq 1$, i.e. $\alpha\geq \frac 1 2$. We reduce this case to the one considered before, i.e. we estimate the double sum in \eqref{pf:IDS:eq0} by
\beq
 \sum_{k\in\Lambda_L}\sum_{m\in\Lambda^c _L} \frac 1 {|k-m|^{d+2\alpha}} \leq  \sum_{k\in\Lambda_L}\sum_{m\in\Lambda^c _L} \frac 1 {|k-m|^{d+1/2}}\leq  C L^{d-1/2},
\eeq
for some constant $C>0$, where the bound in the r.h.s. follows from the analysis before.
Hence, we obtain that
\begin{align}
\eqref{pf:IDS:eq0} \leq
\frac {\wtilde C_{d,\alpha}} {|\Im z|^2} L^{d-\min\{2\alpha,1/2\}},
\end{align}
for some constant $\wtilde C_{d,\alpha}>0$ depending only on $d$ and $\alpha$. Plugging this into \eqref{Pf:Prop:Eq1}, dividing  by $\frac 1 {|\Lambda_L|}$ and taking the limit $L\to\infty$, gives \eqref{Pf:Prop:Eq0} and the assertion follows.
\end{proof}

\section{Proof of Theorem \ref{Theorem}}\label{s:thm_LT}

In this section we first prove Dirichlet-Neumann bracketing for our model. To do so,
we denote by $-\Delta^{N/D}_{\Lambda}$ the restriction of $-\Delta$ to the set $\Lambda\subset\Z^d$ with Neumann, respectively Dirichlet boundary conditions on the boundary of $\Lambda$. For the precise definition of $-\Delta_\Lambda^{N/D}$ we refer to \cite[Sec. 5.2]{MR2509110}. In particular, we write $-\Delta^{N/D}_{L}$ for the restriction to the box $\Lambda_L$ with the respective boundary conditions. Moreover, we write $H^{N/D}_{\alpha,L}$ for $(-\Delta_L^{N/D})^\alpha+ V_{\omega,L}$ with $V_{\omega,L}:= 1_L V_\omega 1_L$. In this section we set the coupling $\lambda=1$. 

\begin{lemma}[Dirichlet-Neumann bracketing]\label{Pf:thm:lm}
Let $0<\alpha\leq 1$ and $\Lambda_1\subset \Lambda_2\subset \Z^d$ be boxes. Then we have the following operator inequalities
\beq\label{Pf:thm:eq0}
1_{\Lambda_2} (-\Delta)^\alpha 1_{\Lambda_2} \leq
\big(-\Delta^D_{\Lambda_2}\big)^\alpha
\leq \big(-\Delta^D_{\Lambda_1}\big)^\alpha\oplus \big(-\Delta^D_{\Lambda_{2}\setminus\Lambda_1}\big)^\alpha
\eeq
and
\beq\label{Pf:thm:eq1}
1_{\Lambda_2} (-\Delta)^\alpha 1_{\Lambda_2} \geq
\big(-\Delta^N_{\Lambda_2}\big)^\alpha
\geq \big(-\Delta^N_{\Lambda_1}\big)^\alpha\oplus \big(-\Delta^N_{\Lambda_{2}\setminus\Lambda_1}\big)^\alpha.
\eeq
\end{lemma}

\begin{proof}
We first note that, using Dirichlet-Neumann bracketing for the discrete Laplacian \cite[Sec. 5.2]{MR2509110}, we obtain
\beq\label{opinequ1}
-\Delta^N_{\Lambda_2}\oplus -\Delta^N_{\Z^d\setminus \Lambda_2}\leq -\Delta\leq -\Delta^D_{\Lambda_2}\oplus -\Delta^D_{\Z^d\setminus \Lambda_2}.
\eeq
The function $x\mapsto x^\alpha$, $x\in[0,\infty)$ is operator monotone for $\alpha\in (0,1]$, i.e. for all self-adjoint operators $0\leq A\leq B$ we obtain
$0\leq A^\alpha\leq B^\alpha$ provided $\alpha\in (0,1]$, see \cite[Thm V.2.10]{MR1477662}.
Applying this to \eqref{opinequ1}, we obtain for $0<\alpha\leq 1$
\beq
\big(-\Delta^N_{\Lambda_2}\big)^\alpha\oplus \big(-\Delta^N_{\Z^d\setminus \Lambda_2}\big)^\alpha\leq (-\Delta)^\alpha\leq \big(-\Delta^D_{\Lambda_2}\big)^\alpha\oplus \big( -\Delta^D_{\Z^d\setminus \Lambda_2}\big)^\alpha
\eeq
and multiplying the above with $1_{\Lambda_2}$ from both sides implies
\beq\label{lm0_eq0}
\big(-\Delta^N_{\Lambda_2}\big)^\alpha\leq 1_{\Lambda_2} (-\Delta)^\alpha 1_{\Lambda_2} \leq \big(-\Delta^D_{\Lambda_2}\big)^\alpha.
\eeq
Similarly, Dirichlet-Neumann bracketing for the discrete Laplacian on $\Lambda_2= \Lambda_1\cup {\Lambda_{2}\setminus\Lambda_1}$ and  the operator monotonicity of $x\mapsto x^\alpha$ also imply that
\beq\big(-\Delta^N_{\Lambda_1}\big)^\alpha \oplus \big( -\Delta^N_{\Lambda_{2}\setminus\Lambda_1}\big)^\alpha
\leq
\big(-\Delta^{N/D}_{\Lambda_2} \big)^\alpha
\leq
 \big(-\Delta^D_{\Lambda_1}\big)^\alpha \oplus \big( -\Delta^D_{\Lambda_{2}\setminus\Lambda_1}\big)^\alpha.
\eeq
The latter together with inequality \eqref{lm0_eq0} gives the assertion.
\end{proof}

\begin{proof}[Proof of Theorem \ref{Theorem}]
The definition of the IDS $N_\alpha$ given in Proposition \ref{Proposition} and inequalities \eqref{Pf:thm:eq0} and \eqref{Pf:thm:eq1} in Lemma \ref{Pf:thm:lm} imply for all $L\in\N$
\beq\label{Pf:thm:eq2}
\frac 1 {|\Lambda_L|} \E\big[ \Tr 1_{(-\infty,E]}\big(H_{\alpha,L}^D\big) \big]
\leq
N_\alpha(E)
\leq
\frac 1 {|\Lambda_L|} \E\big[ \Tr 1_{(-\infty,E]}\big(H_{\alpha,L}^N\big) \big].
\eeq
 Now, basically the rest of the proof follows along the same lines as in \cite[Sec. 6]{MR2509110} where one has to choose the length scale $L\sim E^{\frac 1 {2\alpha}}$.

 For the reader's convenience, we sketch the proof of the upper bound in \eqref{Pf:thm:eq2} to see the emergence of the exponent $ d/(2\alpha)$. We have that
\begin{align}
\frac 1 {|\Lambda_L|} \E\big[ \Tr 1_{(-\infty,E]}\big(H^N_{\alpha,L}\big) \big]
&\leq
\mathbb P\big[ E_0\big(H^N_{\alpha,L}\big)  < E]\label{Pf:thm:eq3}
\end{align}
where $E_0(A)$ denotes the lowest eigenvalue and $E_1(A)$ denotes in the following the next bigger eigenvalue of a self-adjoint matrix $A$.
We aim at using Temple's inequality, see \cite[Lem. 6.3]{MR2509110}. To do so, we first perform some auxiliary calculations. A direct calculation shows that
$
E_1\big( -\Delta_L^N\big)\geq  c L^{-2}
$
for some constant $c>0$
and therefore
\beq\label{Pf:thm:eq4}
E_1\big( H^N_{\alpha,L}\big)\geq E_1\big( (-\Delta_L^N)^\alpha\big) \geq  c^\alpha L^{-2\alpha}.
\eeq
Next we set $\wtilde\omega(j) :=\min\big\{\omega(j), \frac {c^\alpha} 3 L^{-2\alpha}\big\}$ and define
\beq
\wtilde V_{\omega,L}= \sum_{j\in\Lambda_L} \wtilde\omega(j) |\delta_j\>\<\delta_j|\qquad\text{and}\qquad
\wtilde H^N_{\alpha,L}= (-\Delta^N_L)^\alpha + \wtilde V_{\omega,L}.
\eeq
Let $\psi\in\ell^2(\Z^d)$ be given by $\psi(n) = \frac 1 {|\Lambda_L|^{1/2}}$ for $n\in\Z^d$. Then have
\beq
\big\<\psi,\wtilde V_{\omega,L}\psi\big\>\leq \frac{c^\alpha} 3 L^{-2\alpha}.
\eeq
This, combined with \eqref{Pf:thm:eq4}, gives the estimate
\beq
E_1\big(\wtilde H^N_{\alpha,L}\big)
\geq
E_1\big((-\Delta^N_L)^\alpha \big)
\geq
\frac{c^\alpha} 3 L^{-2\alpha}
\geq
\big\<\psi, \wtilde H^N_{\alpha,L}\psi\big\>,
\label{inequalities}
\eeq
where for the last inequality we used the fact that $\<\psi, (-\Delta^N_L)^\alpha\psi\>=0$.
Hence, we are in position to apply Temple's inequality which implies the lower bound
\begin{align}
E_0\big(\wtilde H^N_{\alpha,L}\big)
&\geq
\<\psi,\wtilde H^N_{\alpha,L}\psi\> - \frac{\<\psi,(\wtilde H^N_{\alpha,L})^2\psi\>}{c^{\alpha}L^{-2\alpha}-\<\psi,\wtilde H^N_{\alpha,L}\psi\> }\notag\\
&\geq
\frac 1 {|\Lambda_L|} \sum_{j\in\Lambda_L} \wtilde \omega(j) \big( 1- \frac{3}{2 c^{\alpha}L^{-2\alpha}}  \frac {c^\alpha} 3 L^{-2\alpha}
\big)
= \frac 1 2 \frac 1 {|\Lambda_L|} \sum_{j\in\Lambda_L} \wtilde \omega(j), \label{Pf:thm:eq5}
\end{align}
where we used for the second inequality $\<\psi, (-\Delta^N_L)^\alpha\psi\>=0$ and \eqref{inequalities}.
This, \eqref{Pf:thm:eq2} and \eqref{Pf:thm:eq3} imply for all $L>0$ and $E\in\R$ that
\beq
N_\alpha(E)
\leq
\mathbb P\big[ E_0\big(H^N_{\alpha,L}\big)  < E \big]
\leq
\mathbb P\big[ E_0\big(\wtilde H^N_{\alpha,L}\big)  < E \big]
\leq
\mathbb P\big[  \frac 1 {|\Lambda_L|} \sum_{j\in\Lambda_L} \wtilde \omega(j)  < 2E\big].
\eeq
We choose the length scale $L>0$ according to
\beq
L:= \lfloor \beta E^{-\frac 1 {2\alpha}}\rfloor
\eeq
for some $\beta>0$ small. Then a large deviation estimate \cite[Lem. 6.4] {MR2509110} implies that there exist $\gamma'>\gamma>0$ such that
\begin{align}
\mathbb P\big[  \frac 1 {|\Lambda_L|} \sum_{j\in\Lambda_L} \wtilde \omega(j)  < 2E\big]
\leq e^{-\gamma|\Lambda_L|}
\leq e^{-\gamma (2 \lfloor \beta E^{-\frac 1 {2\alpha}}\rfloor ^d+1)}
\leq e^{-\gamma' E^{-\frac d {2\alpha}}}.
\end{align}
This gives the desired upper bound taking the double logarithm.

Using the assumption $P_0([0,\varepsilon))\geq C \varepsilon^\kappa$ for some $C,\kappa>0$,
the lower bound follows from analysing $\frac 1 {|\Lambda_L|} \E\big[ \Tr 1_{(-\infty,E]}\big(H_{\alpha,L}^D\big) \big]$ in the same way as was done in \cite[Sec. 6.3]{MR2509110}. We note that one has to choose again the length scale $L\sim E^{\frac 1 {2\alpha}}$.
\end{proof}

\section{Proof of Theorem \ref{Lemma}}\label{s:off_diag_decay}

\begin{proof}[Proof of Theorem \ref{Lemma} \textit{(i)}]
Let $0<\alpha<1$. We use the identity
\beq\label{eq:ExpRep}
(-\Delta)^\alpha =-  \frac 1{|\Gamma(-\frac\alpha 2)|} \int_0^\infty\d t\,t^{-1-\alpha/2}\big( e^{t\Delta} - \id\big),
\eeq
where the integral converges in operator norm, see \cite[Thm 1.1 (c)]{MR3613319}. Hence, for $n,m\in\Z^d$ with $n\neq m$
\beq\label{thm:eq00}
(-\Delta)^\alpha(n,m) =- \frac 1{|\Gamma(-\frac\alpha 2)|} \int_0^\infty\d t\,t^{-1-\alpha/2} e^{t\Delta}(n,m).
\eeq
Let $-\Delta_1$ be the one-dimensional discrete Laplacian, i.e. acting on $\ell^2(\Z)$. Since $-\Delta = (-\Delta_1)\otimes \cdots \otimes \id +\, ...\, + \id \otimes \cdots \otimes (-\Delta_1)$ acting on $\ell^2(\Z^d) = \otimes_{k=1}^d \ell^2(\Z)$ we obtain 
\beq\label{thm:eq11}
e^{t\Delta}(n,m) = \prod_{k=1}^d e^{t\Delta_1} (n_k,m_k)
\eeq
with $n=(n_1,\cdots,n_d)$ and $m=(m_1,\cdots,m_d)$. The matrix elements of $e^{t \Delta_1}$ can be computed explicitly \cite[Eq. (2.2)]{MR3787555} 
\beq
e^{t\Delta_1} (n_1,m_1) = e^{-2 t} I_{n_1-m_1}(2t)
\eeq
where $n_1,m_1\in\Z$ and $I_{(\cdot)}$ stands for the modified Bessel function. The modified Bessel function is symmetric, i.e. $I_k(s)=I_{-k}(s)$ for all $k\in\N$ and $s\geq 0$, see e.g. \cite[Eq. 8.486]{gradshteyn:2007} and therefore
\beq\label{eq.bessel}
e^{t\Delta_1} (n_1,m_1) = e^{-2 t} I_{|n_1-m_1|}(2t).
\eeq
The modified Bessel function $I_k$ has the following integral representation for $k>-1/2$ and $s\geq 0$
\beq\label{bessel}
I_{k}(s) =  \frac{(s/2)^k}{\Gamma(k+1/2)\Gamma(1/2)} \int_{-1}^1\d t\, (1- t^2)^{k-1/2} \cosh(s t) > 0
\eeq
 see e.g. \cite[Eq. 8.431]{gradshteyn:2007} which shows positivity of the modified Bessel function. Equation \eqref{bessel} with $k=|n_1-m_1|$ implies
 that $e^{-2 t} I_{|n_1-m_1|}(2t)>0$ for all $t\geq 0$. Inserting this in \eqref{eq.bessel}, \eqref{thm:eq11} and subsequently in \eqref{thm:eq00} implies the result for $n\neq m \in \Z^d$. 

The positivity of $(-\Delta)^\alpha(n.n)$ follows directly from the integral representation \eqref{eq:ExpRep} using that $0\leq e^{t\Delta}<1$ for all $t>0$ and therefore $e^{t\Delta}(n,n)< 1$ for all $t>0$. 
\end{proof}

\begin{proof}[Proof of Theorem \ref{Lemma} \textit{(ii)}, \textit{(iii)}]
We recall that the discrete Fourier transform diagonalizes the discrete Laplace operator, i.e. $(\mathcal F (-\Delta) \mathcal F^*)(k) = \sum_{j=1}^d \big( 2- 2 \cos(k_j)\big)$, where $k=(k_1,...k_d)\in \mathbb T^d$ and $\mathbb T^d=[-\pi,\pi]^d$ is the $d$-dimensional torus. Here, $\mathcal F:\ell^2(\Z^d)\to L^2(\mathbb T^d)$ is the discrete Fourier transform given by $(\mathcal F u)(k)=\displaystyle \frac 1 {(2\pi)^{d/2}}\sum_{n\in\Z^d} u(n) e^{-i n k}$, where $u\in \ell^2(\Z^d)$ and $k\in \mathbb T^d$. Hence, for $m,n\in\Z^d$
\beq
(-\Delta)^\alpha(n,m) = \frac 1 {(2\pi)^d}\int_{[-\pi,\pi ]^d} \d k \, \Big(\sum_{j=1}^d \big( 2- 2 \cos(k_j)\big)\Big)^\alpha e^{-i(n-m) k}.
\eeq
Let $\psi\in C_c^\infty(\R^d)$ be such that
\begin{itemize}
\item[(i)] $\supp \psi \subset B_1(0)$, where $B_1(0)$ is the ball of radius $1$ about the origin.
\item[(ii)] $0\leq \psi \leq 1$ and $\psi(x)=1$ for all $x\in B_{1/2}(0)$.
\end{itemize}
Then, we define the function
\beq
\Psi_\alpha(k):= \Big(\sum_{j=1}^d \big( 2- 2 \cos(k_j)\big)\Big)^\alpha\big( 1 - \psi(k) \big)
\eeq
and note that $\Psi_\alpha  \in C^{\infty}\big(\mathbb T^d\big)$. 
This implies that the Fourier coefficients of $\Psi_\alpha$ satisfy for $z\in\Z^d$
\beq\label{eq:decayFou}
 \int_{[-\pi,\pi]^d} \d k\,  \Psi_\alpha(k) e^{-i z k} = O\big(|z|^{-\infty}\big) ,
\eeq
as $|z|\to\infty$, i.e.
 the Fourier coefficients of $\Psi_\alpha$ decay faster than any power, see e.g. \cite[Sec. 3]{MR2445437}. Since $(-\Delta)^\alpha(n,m)$ depends on $n-m$ only we set $m=0$ in the following. 
In particular, \eqref{eq:decayFou} implies
\begin{align}
&\lim_{|n|\to\infty} -|n|^{d+2\alpha} (-\Delta)^\alpha(n,0) \notag\\
& =\lim_{|n|\to\infty} -|n|^{d+2\alpha} \frac 1 {(2\pi)^d}
 \int_{[-\pi,\pi]^d} \d k\,  \Big(\sum_{j=1}^d \big( 2- 2 \cos(k_j)\big)\Big)^\alpha \psi(k) e^{-i n k} .\label{lm:pf:eq0}
\end{align}
We note that the last integral can be seen as a Fourier transform of a smooth function on $\R^d\setminus\{0\}$ with a cusp at $0$. We further rewrite
\begin{align}
\int_{[-\pi,\pi]^d} \d k\,  \Big(\sum_{j=1}^d \big( 2- 2 \cos(k_j)\big)\Big)^\alpha \psi(k) e^{-in k}
&=
\int_{\R^d} \d k\, |k|^{2\alpha} \Phi(k) e^{-i n k},
\end{align}
where for $k\in\R^d$
\beq
\Phi(k) := \frac{\Big(\sum_{j=1}^d \big( 2- 2 \cos(k_j)\big)\Big)^\alpha}{|k|^{2\alpha}}\psi(k).
\eeq
We have that $\Phi\in C_c^\infty(\R^d)$ which follows from the fact that $\psi\in C_c^\infty([-\pi,\pi]^d)$,  $\varphi\in C^\infty([-\pi,\pi]^d)$ with $\varphi(k):=\frac{\sum_{j=1}^d ( 2- 2 \cos(k_j))}{|k|^2}>0$ for all $k\in[-\pi,\pi]^d$ and $x\mapsto x^\alpha\in C^\infty((0,\infty))$. Moreover, one easily sees that $\Phi(0)=1$. 

We denote the continuous  fractional negative Laplacian on $\R^d$ by $(-\Delta_c)^\alpha$. Its Fourier representation implies for $w\in\R^d$
\beq
\frac 1 {(2\pi)^{d}}\int_{\R^d} \d k\, |k|^{2\alpha} \Phi(k) e^{-iw k} = \frac 1 {(2\pi)^{d/2}}\big((-\Delta_c)^\alpha \mathcal F_c^* \Phi\big) (-w),
\eeq
where $\mathcal F_c^*$ is the inverse (continuous) Fourier transform. Since $\mathcal F_c^* \Phi\in\mathcal S(\R^d)$, where $\mathcal S(\R^d)$ is the set of  Schwartz functions, we obtain using Lemma \ref{lem:cont_lap} 
\beq
\lim_{|w|\to\infty}-\frac {|w|^{d+2\alpha}} {(2\pi)^{d/2}} \big((-\Delta_c)^\alpha \mathcal F_c^* \Phi\big) (w) = \frac{K_{d,\alpha} }{(2\pi)^{d/2}}\int_{\R^d} \d y\, \big(\mathcal F_c^*\Phi\big)(y),
\eeq
where $K_{d,\alpha}$ is given in Lemma \ref{lem:cont_lap} below. Now
\beq
\frac{1 }{(2\pi)^{d/2}}\int_{\R^d} \d y\, \big(\mathcal F_c^*\Phi\big)(y) = \big(\mathcal F_c \mathcal F_c^*\Phi\big)(0) =\Phi(0)= 1.
\eeq
Altogether we end up with 
\beq
\lim_{|n-m|\to\infty} |n-m|^{d+2\alpha} \big(-(-\Delta)^\alpha(n,m)\big)
 = 
K_{d,\alpha}>0
\eeq
which is \textit{(ii)}.

The strict positivity $-(-\Delta)^\alpha(n,m)>0$ for all $n\neq m\in\Z^d$ proved in Theorem \ref{Lemma} \textit{(i)} together with part \textit{(ii)} implies part \textit{(iii)} directly.
\end{proof}

\appendix

\section{A limit of the continuous fractional Laplacian}

A crucial ingredient to the proof of Theorem \ref{Lemma} is understanding the behaviour of the continuous fractional Laplacian $(-\Delta_c)^\alpha$ applied to Schwartz functions for large $x\in\R^d$. 

\begin{lemma}\label{lem:cont_lap}
Let $\mathcal S(\R^d)$ be the set of Schwartz functions, $\varphi\in \mathcal S(\R^d)$, $0<\alpha<1$ and $(-\Delta_c)^\alpha$ be the continuous fractional negative Laplacian. Then 
\beq
\lim_{|x|\to\infty} |x|^{d+2\alpha} \big((-\Delta_c)^\alpha \varphi\big)(x) = - K_{d,\alpha} \int_{\R^d} \d y\, \varphi(y)
\eeq
where
\beq
K_{d,\alpha}= \frac{4^\alpha \Gamma(d/2+\alpha)}{\pi^{d/2} |\Gamma(-\alpha)|}>0. 
\eeq
\end{lemma}

\begin{proof}
Let $\varphi\in \mathcal S(\R^d)$,  $0<\alpha<1$ and $x\in \R^d$. 
 Then the continuous fractional negative Laplacian can be written as
\beq\label{eq:int_lap}
\big((-\Delta_c)^\alpha \varphi\big)(x) = K_{d,\alpha} \lim_{b\to 0} \int_{\R^d\setminus B_b(x)} \d y\, \frac{\big(\varphi(x)-\varphi(y)\big)}{|x-y|^{d+2\alpha}}
\eeq
with $K_{d,\alpha}>0$ given above,
see \cite{MR3613319}. 

Let $x\neq 0$ and $0<b<\frac{|x|} 2$. We split the integral in \eqref{eq:int_lap} into two parts as follows:
\begin{align}\label{eq:app0}
&\int_{\R^d\setminus B_b(x)} \d y\, \frac{\big(\varphi(x)-\varphi(y)\big)}{|x-y|^{d+2\alpha}} \notag\\
&=
\int_{B_{\frac{|x|}2}(x)\setminus B_b(x)} \d y\, \frac{\big(\varphi(x)-\varphi(y)\big)}{|x-y|^{d+2\alpha}}
+
\int_{B^c_{\frac{|x|}2}(x)} \d y\, \frac{\big(\varphi(x)-\varphi(y)\big)}{|x-y|^{d+2\alpha}}.
\end{align}
We first consider the second part
\begin{align}
\int_{B^c_{\frac{|x|}2}(x)} \d y\, \frac{\big(\varphi(x)-\varphi(y)\big)}{|x-y|^{d+2\alpha}}
=& \int_{B^c_{\frac{|x|}2}(0)} \d y\, \frac {\varphi(x)} {|y|^{d+2\alpha}} -  \int_{B^c_{\frac{|x|}2}(x)} \d y \frac{\varphi(y)}{|x-y|^{d+2\alpha}}.
\end{align}
Since $\varphi\in \mathcal S(\R^d)$ decays faster than any polynomial the above implies
\begin{align}\label{eq:app1}
\lim_{|x|\to\infty} |x|^{d+2\alpha} \int_{B^c_{\frac{|x|}2}(x)} \d y\, \frac{\big(\varphi(x)-\varphi(y)\big)}{|x-y|^{d+2\alpha}}
=
-\lim_{|x|\to\infty} 
\int_{B^c_{\frac{|x|}2}(x)} \d y \frac{|x|^{d+2\alpha} }{|x-y|^{d+2\alpha}} \varphi(y).
\end{align}
One directly sees that for all $y\in\R^d$
\beq
\lim_{|x|\to\infty} 1_{B^c_{\frac{|x|}2}(x)}(y) \frac{|x|^{d+2\alpha} }{|x-y|^{d+2\alpha}} = 1
\eeq
 where $1_S$ stands here for the indicator function of a set $S\in \text{Borel}(\R^d)$ and for all $x,y\in \R^d$
 \beq
\Big| 1_{B^c_{\frac{|x|}2}(x)}(y) \frac{|x|^{d+2\alpha} }{|x-y|^{d+2\alpha}} \Big| \leq 2^{d+2\alpha}.
 \eeq
Hence the dominated convergence theorem implies 
\beq\label{eq:app2}
\lim_{|x|\to\infty} 
\int_{B^c_{\frac{|x|}2}(x)} \d y \frac{|x|^{d+2\alpha} }{|x-y|^{d+2\alpha}} \varphi(y)
= 
\int_{\R^d} \d y\, \varphi(y). 
\eeq

For the remaining term in \eqref{eq:app0} we compute for $0<b<\frac {|x|} 2$,
\begin{align}
\int_{B_{\frac{|x|}2}(x)\setminus B_b(x)} \d y \frac{\big(\varphi(x)-\varphi(y)\big)}{|x-y|^{d+2\alpha}}
&=\int_b^{\frac{|x|} 2} \d r r^{d-1}\int_{\partial B_1(0)} \d S(w) \frac{\big(\varphi(x)-\varphi(x+rw )\big)}{r^{d+2\alpha}}\notag\\
&
=\int_b^{\frac{|x|} 2} \d r\, \frac{r^{d-1}}  {r^{d+2\alpha}} \big(\phi(0)-\phi(r)\big),
\end{align}
where $\phi(r) := \int_{\partial B_1(0)} \d S(w)\,  \varphi(x+rw )$ and we denote by $\d S$ integration with respect to surface measure. Then
\begin{align}
\phi'(r) &= \int_{\partial B_1(0)} \d S(w)\, \big((\nabla \varphi)(x+rw)\big)\cdot  w\notag\\
&=\frac 1 {r^{d-1}} \int_{\partial B_r(x)} \d S(v)\, \big((\nabla \varphi)(v)\big)\cdot  \frac {v-x} r\notag\\
&=\frac 1 {r^{d-1}} \int_{\partial B_r(x)} \d S(v)\, \frac{\partial \varphi}{\partial \nu}(v),
\end{align}
where $\nu$ is the outer normal vector. Green's formula implies that
\begin{align}
\int_{\partial B_r(x)} \d S(v)\, \frac{\partial \varphi}{\partial \nu}(v) = \int_{B_r(x)} \d y\, (\Delta \varphi) (y).
\end{align}
and therefore we obtain
\begin{align}
\int_b^{\frac{|x|} 2} \d r\, \frac {r^{d-1}} {r^{d+2\alpha}} \big(\phi(0)-\phi(r)\big) = -\int_b^{\frac{|x|} 2} \d r\, \frac 1 {r^{2\alpha+1}}\int_0^r \d s  \frac 1 {s^{d-1}}\int_{B_s(x)} \d y\, (\Delta \varphi) (y).
\end{align}
This implies the bound
\begin{align}
&\Big|\int_{B_{\frac{|x|}2}(x)\setminus B_b(x)} \d y\, \frac{\big(\varphi(x)-\varphi(y)\big)}{|x-y|^{d+2\alpha}}\Big|\notag\\
&\leq
C_{d}^{(3)}
\sup_{y\in B_{\frac{|x|}2}(x)} |(\Delta \varphi)(y)|
\int_b^{\frac{|x|} 2} \d r\, \frac 1 {r^{2\alpha+1}}\int_0^r \d s\, s\notag\\
&=
\frac{C_{d}^{(3)}}{2}
\sup_{y\in B_{\frac{|x|}2}(x)} |(\Delta \varphi)(y)|
\int_b^{\frac{|x|} 2} \d r\, r^{1-2\alpha}
\end{align}
for some constant $C_{d}^{(3)}>0$ depending on the dimension only. Therefore, we end up for fixed $0\neq x\in\R^d$ with
\begin{align}\label{eq:app35}
\limsup_{b\to 0} \Big|\int_{B_{\frac{|x|}2}(x)\setminus B_b(x)} \d y\, \frac{\big(\varphi(x)-\varphi(y)\big)}{|x-y|^{d+2\alpha}}\Big|
\leq C_{d}^{(4)}
\sup_{y\in B_{\frac{|x|}2}(x)} |(\Delta \varphi)(y)| |x|^{2-2\alpha}
\end{align}
for some constant $C_{d}^{(4)}>0$. Since $\varphi\in \mathcal S(\R^d)$, we have 
\beq\label{eq:app4}
\limsup_{|x|\to \infty}\ C_{d}^{(4)}
\sup_{y\in B_{\frac{|x|}2}(x)} |(\Delta \varphi)(y)| |x|^{2-2\alpha} = 0.
\eeq

Equations \eqref{eq:app0}, \eqref{eq:app1}, \eqref{eq:app2}, \eqref{eq:app35} and \eqref{eq:app4} imply
\begin{align}
\lim_{b\to 0}\int_{\R^d\setminus B_b(x)} \d y\, \frac{\big(\varphi(x)-\varphi(y)\big)}{|x-y|^{d+2\alpha}}
=
-\int_{\R^d} \d y\, \varphi(y)
\end{align}
which inserted in \eqref{eq:int_lap} gives the result. 
\end{proof}

\section*{Acknowledgements}

The authors thank for the financial support of the Mathematisches Institut of the Heinrich-Heine-Universit\"at D\"usseldorf, where this work was initiated.

\noindent
The authors would like to thank the anonymous referee for carefully reading our manuscript and for the suggestions that led to the improvement of a previous version of Theorem \ref{Lemma} , in the form of the lower bound in \eqref{eq:decay}.

\newcommand{\etalchar}[1]{$^{#1}$}

\end{document}